\documentclass[11pt]{article}
\usepackage{amscd}
\usepackage{amsfonts}
\usepackage{amsmath}
\usepackage{amssymb}
\usepackage{amsthm}
\usepackage{bbm}
\usepackage{CJK}
\usepackage{fancyhdr}
\usepackage{graphicx}
\usepackage{hyperref}
\usepackage{indentfirst}
\usepackage{latexsym}
\usepackage{mathrsfs}
\usepackage{xypic}
\usepackage{amsmath,amssymb,amsfonts,amsthm,fancyhdr}
\usepackage{epsfig,graphicx,picins,picinpar,subfigure}
\usepackage{pstricks}
\usepackage{amssymb}
\usepackage{xypic}
\usepackage[compress]{cite}

\newtheorem{theorem}{Theorem}[section]
\newtheorem{prop}[theorem]{Proposition}
\newtheorem{lemma}[theorem]{Lemma}
\newtheorem{coro}[theorem]{Corollary}
\newtheorem{prop-def}{Proposition-Definition}[section]

\theoremstyle{definition}
\newtheorem{defn}[theorem]{Definition}

\newtheorem{remark}[theorem]{Remark}
\newtheorem{exam}[theorem]{Example}

\usepackage[top=1in,bottom=1in,left=1.25in,right=1.25in]{geometry}
\def\<{\langle}
\def\>{\rangle}

\date{\today}
\begin{document}
\renewcommand{\baselinestretch}{1.2}
\renewcommand{\arraystretch}{1.0}
\title{\bf  Cohomologies of  modified $\lambda$-differential Lie triple systems and applications}
\author{{\bf Wen Teng$^{*}$,   Fengshan Long$^{}$, Yu Zhang$^{}$}\\
{\small  School of Mathematics and Statistics, Guizhou University of Finance and Economics} \\
{\small  Guiyang  550025, P. R. of China}\\
  {\small * Corresponding author: tengwen@mail.gufe.edu.cn (Wen Teng)} \\
}

 \maketitle
\begin{center}
\begin{minipage}{13.cm}

{\bf Abstract}
In this paper, we introduce the concept and representation of modified $\lambda$-differential Lie triple systems. Next, we define the cohomology of modified $\lambda$-differential Lie triple systems with coefficients in a suitable representation.  As   applications of the proposed cohomology theory,  we study 1-parameter formal deformations and abelian extensions of modified $\lambda$-differential Lie triple systems.
 \smallskip

 \smallskip

 \end{minipage}
 \end{center}
 \normalsize\vskip0.5cm

\section{Introduction}
\def\theequation{\arabic{section}. \arabic{equation}}
\setcounter{equation} {0}

Jacobson  \cite{Jacobson1,Jacobson2} introduced the concept of Lie triple systems by quantum mechanics and Jordan theory. In fact, Lie triple systems originated from E.Cartan's research on symmetric spaces and totally geodesic submanifolds. Since then, the structure theory,  representation theory, cohomology theory, deformation theory and extension theory of Lie triple systems   were established in \cite{Lister,Yamaguti,Hodge, Harris,Kubo,Lin,Zhang}.

Derivations are useful for the study of algebraic structure. Derivations also play an important role in the study of homotopy  algebras, deformation formulas, differential Galois theory, control theory and gauge theories of quantum field theory, see \cite{Voronov,Coll,Magid,Ayala1,Ayala2,Batalin}.
Recently, associative algebras with derivations \cite{Loday}, Lie algebras with derivations \cite{Tang}, Leibniz algebras with derivations \cite{Das1}, Leibniz triple systems with derivations \cite{Wu1} and Lie triple systems with derivations \cite{Wu2,Sun, Guo4} have been widely studied. All these results provide a good starting point for our further study.


In recent years, due to the outstanding work of  \cite{Bai,Guo0,Guo1,Guo2,Das2, Wang}, more and more scholars  begun to pay   attention to the  structure  with arbitrary weights.
 Rota-Baxter Lie algebras of any weight \cite{Das3},  Rota-Baxter 3-Lie algebras of any weight \cite{Hou,Guo3} and  Rota-Baxter Lie triple systems of any weight \cite{Chen} appear successively.
After that,
for $\lambda\in \mathbb{K}$,   the cohomology,  extension and deformation theory of  Lie algebras with differential operators of weight $\lambda$ are   introduced by Li and Wang \cite{Li1}. In addition,  the cohomology and deformation theory of modified Rota-Baxter associative algebras and modified Rota-Baxter Leibniz algebras of weight $\lambda$ are given in \cite{Das4,Li2,Mondal}.
  The concept of  modified   $\lambda$-differential Lie algebras  are  introduced in \cite{Peng}.
The method of this paper is to follow the recent work \cite{Wu2,Sun,Guo2, Guo4,Peng}.
Our main objective is to  consider  modified $\lambda$-differential Lie triple systems.
  More precisely,  we  introduce the concept of  a modified $\lambda$-differential Lie triple system, which includes a Lie triple system and a modified $\lambda$-differential operator.
We  define a cochain map $\Phi$, and then give the cohomology   of modified $\lambda$-differential Lie triple systems with coefficients in a representation by using a Yamaguti coboundary operator $\delta$ and a cochain map $\Phi$.  Next we  study 1-parameter  formal deformations  of a modified $\lambda$-differential Lie triple system using the third cohomology group of  modified $\lambda$-differential Lie triple system with the coefficient in the adjoint representation.
Finally, we study abelian extensions   of a modified $\lambda$-differential Lie triple system using the third cohomology group.   All the results in this paper can be regarded as  generalizations of Lie triple systems with derivations  \cite{Wu2,Sun, Guo4}.

The paper is organized as follows.   In Section  \ref{sec:Repre}, we  introduce the  concept of  a modified $\lambda$-differential Lie triple system, and give its representation.  In Section \ref{sec:cohomologydo},   we define a cohomology theory for modified $\lambda$-differential Lie triple systems with coefficients in a representation.
 In Section \ref{sec:def},  we  study 1-parameter  formal deformations  of a modified $\lambda$-differential Lie triple system.
 In Section \ref{sec:ext},  we study abelian extensions   of a modified $\lambda$-differential Lie triple system.

Throughout this paper, $\mathbb{K}$ denotes a field of characteristic zero. All the algebras,  vector spaces,
algebras, linear maps and tensor products are taken over $\mathbb{K}$.

\section{Representations of modified $\lambda$-differential Lie triple systems}\label{sec:Repre}
\def\theequation{\arabic{section}.\arabic{equation}}
\setcounter{equation} {0}

In this section, first, we  recall some basic concepts of    Lie triple systems from \cite{Jacobson1} and \cite{Yamaguti}.
Then, we introduce the  concept of a  modified $\lambda$-differential  Lie triple system    and its representation.

\begin{defn}  \cite{Jacobson1}
(i) A Lie triple system is a   pair $(\mathfrak{L}, [\cdot, \cdot, \cdot])$ in which $\mathfrak{L}$ is a vector space together with  a ternary operation $[\cdot, \cdot, \cdot]$
on $\mathfrak{L}$   such that
\begin{align}
&[x,y,z]+[y,x,z]=0,\label{2.1}\\
&[x,y,z]+[z,x,y]+[y,z,x]=0,\label{2.2}\\
 &[a, b, [x, y, z]]=[[a, b, x], y, z]+ [x,  [a, b, y], z]+ [x,  y, [a, b, z]],\label{2.3}
\end{align}
for all $ x, y, z, a, b\in \mathfrak{L}$.\\
(ii) A homomorphism between two  Lie triple systems  $(\mathfrak{L}_1, [\cdot, \cdot, \cdot]_1)$ and $(\mathfrak{L}_2, [\cdot, \cdot,\cdot]_2)$ is a linear map $\zeta: \mathfrak{L}_1\rightarrow \mathfrak{L}_2$ satisfying
$\zeta([x, y, z]_1)=[\zeta(x), \zeta(y),\zeta(z)]_2,~~\forall ~x,y,z\in \mathfrak{L}_1.$
\end{defn}

\begin{defn} \cite{Yamaguti}
A representation of  a  Lie triple system $(\mathfrak{L}, [\cdot, \cdot,\cdot])$ on  a vector space $\mathfrak{V}$ is a bilinear map $\theta: \mathfrak{L}\times \mathfrak{L}\rightarrow \mathrm{End}(\mathfrak{V})$, such that
\begin{align}
& \theta(a,b)\theta(x,y)-\theta(y,b)\theta(x,a)-\theta(x, [y,a,b])+D(y,a)\theta(x,b)=0,\label{2.4}\\
& \theta(a,b)D(x,y)-D(x,y)\theta(a,b)+\theta([x,y,a],b)+\theta(a,[x,y,b])=0,\label{2.5}\
\end{align}
for all  $x,y,a,b\in \mathfrak{L}$, where $D(x,y)=\theta(y,x)-\theta(x,y)$. In this case, we also call $\mathfrak{V}$ a $\mathfrak{L}$-module.
\end{defn}

\begin{exam}
Any Lie triple system $(\mathfrak{L}, [\cdot, \cdot,\cdot])$ is a representation over itself with
$$\mathcal{R}: \mathfrak{L}\times \mathfrak{L} \rightarrow \mathrm{End}(\mathfrak{L}), (a,b)\mapsto (x\mapsto [x,a,b]).$$
It is called the adjoint representation over the  Lie triple system.
\end{exam}

\begin{defn}
(i)  Let  $\lambda\in \mathbb{K}$ and  $(\mathfrak{L},  [\cdot, \cdot, \cdot])$ be a Lie triple system. A modified $\lambda$-differential operator (also called a modified differential operator of weight $\lambda$) on $\mathfrak{L}$   is a linear operator $d:\mathfrak{L}\rightarrow \mathfrak{L}$,  such that
\begin{align}
d([a,b,c])=&[d(a),b,c]+[a,d(b),c]+[a,b,d(c)]+\lambda[a,b,c]\label{2.6}.
\end{align}
(ii) A modified $\lambda$-differential  Lie triple system (also called
 a modified differential Lie triple system  of weight $\lambda$) is a triple  $(\mathfrak{L}, [\cdot, \cdot,\cdot], d)$  consisting of a
Lie triple system  $(\mathfrak{L}, [\cdot, \cdot,\cdot])$ and a modified $\lambda$-differential operator $d$.\\
(iii) A homomorphism between two modified $\lambda$-differential  Lie triple systems  $(\mathfrak{L}_1, [\cdot, \cdot, \cdot]_1,d_1)$ and $(\mathfrak{L}_2,   [\cdot,\cdot, \cdot]_2, d_2)$  is a Lie triple system  homomorphism $\zeta: (\mathfrak{L}_1,  [\cdot, \cdot, \cdot]_1)\rightarrow (\mathfrak{L}_2,  [\cdot,\cdot, \cdot]_2)$ such that
$\zeta\circ d_1=d_2\circ\zeta$.  Furthermore, if $\zeta$ is  nondegenerate,  then $\zeta$ is called an isomorphism from $\mathfrak{L}_1$ to $\mathfrak{L}_2$.
\end{defn}

\begin{remark}
Let $d$ be a modified $\lambda$-differential operator on $(\mathfrak{L},  [\cdot, \cdot, \cdot])$. If $\lambda=0,$ then $d$ is a derivation on $\mathfrak{L}$. We denote the set of all derivations on $\mathfrak{L}$  by $\mathrm{Der}(\mathfrak{L}).$  See \cite{Zhou} for   various derivations of Lie triple systems.
\end{remark}

Moreover, there is a close relationship between   derivations and modified $\lambda$-differential operators.

\begin{prop}
Let   $(\mathfrak{L}, [\cdot, \cdot, \cdot ])$ be a  Lie triple system.  Then, a linear operator $d:\mathfrak{L}\rightarrow \mathfrak{L}$ is a   modified $\lambda$-differential operator
if and only if $d+\frac{\lambda}{2} \mathrm{id}_\mathfrak{L}$ is a derivation on $\mathfrak{L}$.
\end{prop}

\begin{proof}
Eq. \eqref{2.6}  is equivalent to
$$(d+\frac{\lambda}{2} \mathrm{id}_\mathfrak{L})([a,b,c])=[(d+\frac{\lambda}{2} \mathrm{id}_\mathfrak{L})(a),b,c]+[a,(d+\frac{\lambda}{2} \mathrm{id}_\mathfrak{L})(b),c]+[a,b,(d+\frac{\lambda}{2} \mathrm{id}_\mathfrak{L})(c)].$$
The proposition follows.
\end{proof}

\begin{exam}
Let $(\mathfrak{L},[\cdot,\cdot],d)$ be a modified $\lambda$-differential  Lie algebra  (see \cite{Peng}, Definition 2.5).  Define trilinear map $[\cdot,\cdot,\cdot]:\mathfrak{L}\times \mathfrak{L}\times \mathfrak{L}\rightarrow \mathfrak{L}$ by $[a,b,c]=[[a,b],c], \forall a,b,c\in \mathfrak{L}.$
Then $(\mathfrak{L},[\cdot,\cdot,\cdot],d)$ is a modified $(2\lambda)$-differential  Lie triple system.
\end{exam}

\begin{exam}
Let $(\mathfrak{L},[\cdot,\cdot, \cdot],d)$ be a modified $\lambda$-differential  Lie triple system.
Then, for $k\in\mathbb{K}$, $(\mathfrak{L},[\cdot,\cdot,\cdot],kd)$ is a modified $(k\lambda)$-differential  Lie triple system.
\end{exam}

\begin{exam}
Let $(\mathfrak{L},  [\cdot, \cdot, \cdot])$ be a 2-dimensional Lie triple system $\mathfrak{L}$ with the  basis $\mathfrak{u}_1$ and $\mathfrak{u}_2$ defined by
$$ [\mathfrak{u}_1,\mathfrak{u}_2,\mathfrak{u}_2]=\mathfrak{u}_1. $$
Then the operator $d=\left(
        \begin{array}{cc}
          k & k_1 \\
          0 & k_2 \\
        \end{array}
      \right)$
is a modified $(-2k_2)$-differential operator  on $\mathfrak{L}$, for $k,k_1,k_2\in \mathbb{K}$.
\end{exam}
\begin{exam}
Let $(\mathfrak{L},  [\cdot, \cdot, \cdot])$ be a 4-dimensional  Lie triple system  with a basis $\mathfrak{u}_1$, $\mathfrak{u}_2, \mathfrak{u}_3$ and $\mathfrak{u}_4$ defined by $ [\mathfrak{u}_1,\mathfrak{u}_2,\mathfrak{u}_1]=\mathfrak{u}_4$.
Then,  the operator
$$d=\left(
        \begin{array}{cccc}
          1 & 0 & k_1 & 0 \\
          0 & 1  & k_2 & 0 \\
          0 & 0  & k_3& 0 \\
         0 & 0  & k_4 & k \\
        \end{array}
      \right)$$
is a modified  $(k-3)$-differential  operator  on $\mathfrak{L}$, for $k,k_i\in \mathbb{K},(i=1,2,3,4)$.
\end{exam}

\begin{defn}
  A representation of the modified $\lambda$-differential Lie triple system   $(\mathfrak{L}, [\cdot, \cdot,\cdot],d)$   is a triple $(\mathfrak{V};  \theta, d_\mathfrak{V})$, where $(\mathfrak{V};  \theta)$ is a representation of the Lie triple system  $(\mathfrak{L}, [\cdot, \cdot,\cdot])$  and  $d_\mathfrak{V}$  is a linear operator on $\mathfrak{V}$, satisfying the following equation
\begin{align}
d_\mathfrak{V}(\theta(x,y)v)=\theta(d(x),y)v+\theta(x,d(y))v+\theta(x,y)d_\mathfrak{V}(v)+\lambda\theta(x,y)v, \label{2.7}
\end{align}
for any $x,y\in \mathfrak{L}$ and $ v\in \mathfrak{V}.$
\end{defn}

From  Eq.  \eqref{2.7}, we   get
 \begin{align}
d_\mathfrak{V}(D(x,y)v)=D(d(x),y)v+D(x,d(y))v+D(x,y)d_\mathfrak{V}(v)+\lambda D(x,y)v, \label{2.8}
\end{align}

Obviously, $(\mathfrak{L};    \mathcal{R}, d)$ is a  representation of the modified $\lambda$-differential Lie triple system  $(\mathfrak{L}, [\cdot, \cdot,\cdot],   d)$.

\begin{remark}
Let $(\mathfrak{V};  \theta, d_\mathfrak{V})$ be a representation of the modified $\lambda$-differential Lie triple system $(\mathfrak{L}, [\cdot, \cdot, \cdot ],d)$. If $\lambda=0,$ then $(\mathfrak{V};  \theta, d_\mathfrak{V})$ is a representation of the  Lie triple system with a derivation $(\mathfrak{L}, [\cdot, \cdot, \cdot ],d)$.
One can refer to \cite{Guo4,Sun,Wu2} for more information about Lie triple systems with  derivations.
\end{remark}

Moreover,
the following result finds the relation between representations over modified $\lambda$-differential Lie triple systems and over Lie triple systems with  derivations.

\begin{prop}
Let $(\mathfrak{V};  \theta)$ be a representation of the  Lie triple system $(\mathfrak{L}, [\cdot, \cdot, \cdot ])$.  Then $(\mathfrak{V};  \theta, d_\mathfrak{V})$ is a representation of the modified $\lambda$-differential Lie triple system $(\mathfrak{L}, [\cdot, \cdot, \cdot ],d)$ if and only if $(\mathfrak{V};  \theta, d_\mathfrak{V}+\frac{\lambda}{2} \mathrm{id}_\mathfrak{V})$ is a representation of the  Lie triple system with a derivation $(\mathfrak{L}, [\cdot, \cdot, \cdot ],d+\frac{\lambda}{2} \mathrm{id}_\mathfrak{L})$.
\end{prop}

\begin{proof}
Eq. \eqref{2.7}  is equivalent to
$$(d_\mathfrak{V}+\frac{\lambda}{2} \mathrm{id}_\mathfrak{V})(\theta(x,y)v)=\theta((d+\frac{\lambda}{2} \mathrm{id}_\mathfrak{L})(x),y)v+\theta(x,(d+\frac{\lambda}{2} \mathrm{id}_\mathfrak{L}))(y))v+\theta(x,y)(d_\mathfrak{V}+\frac{\lambda}{2} \mathrm{id}_\mathfrak{V})(v).$$
The proposition follows.
\end{proof}

\begin{exam}
Let $(\mathfrak{V};  \theta)$ be a representation of  the  Lie triple system $(\mathfrak{L}, [\cdot, \cdot, \cdot ])$.
Then, for $k\in \mathbb{K}$, $(\mathfrak{V}; \theta, \mathrm{id}_\mathfrak{V})$ is a representation of the modified $(-2k)$-differential Lie triple system $(\mathfrak{L}, [\cdot, \cdot, \cdot ],k\mathrm{id}_\mathfrak{L})$.
\end{exam}

\begin{exam}
Let $(\mathfrak{V};  \theta, d_\mathfrak{V})$ be a representation of the modified $\lambda$-differential Lie triple system $(\mathfrak{L}, [\cdot, \cdot, \cdot ],d)$.
Then, for $k\in \mathbb{K}$, $(\mathfrak{V};  \theta, k d_\mathfrak{V})$ is a representation of the modified $(k\lambda)$-differential Lie triple system $(\mathfrak{L}, [\cdot, \cdot, \cdot ],k d)$.
\end{exam}

Next we construct the semidirect product in the context of modified $\lambda$-differential Lie triple systems.

\begin{prop}
Let   $(\mathfrak{L}, [\cdot, \cdot, \cdot ],d)$ be a modified $\lambda$-differential Lie triple system  and   $(\mathfrak{V};  \theta, d_\mathfrak{V})$ be a representation of  it.  Then $\mathfrak{L }\oplus \mathfrak{V}$ is a  modified $\lambda$-differential Lie triple system under the following maps:
\begin{align*}
[x+u, y+v, z+w]_{\ltimes}:=&[x, y, z]+D(x, y)(w)- \theta(x, z)(v)+ \theta(y, z)(u),\\
d \oplus d_\mathfrak{V}(x+u):=&d(x)+d_\mathfrak{V}(u),
\end{align*}
for all  $x, y, z\in \mathfrak{L}$ and $u, v, w\in \mathfrak{V}$.
\end{prop}

\begin{proof}
First, as we all know, $(\mathfrak{L} \oplus \mathfrak{V},[\cdot, \cdot, \cdot]_{\ltimes})$ is a Lie triple system. Next,
for any $x,y,z\in \mathfrak{L}, u,v,w\in \mathfrak{V}$, by Eqs. \eqref{2.6}, \eqref{2.7} and \eqref{2.8}, we have
\begin{align*}
&d \oplus d_\mathfrak{V}([x+u,y+v,z+w]_{\ltimes})\\
=&d ([x, y, z])+  d_\mathfrak{V}( D(x, y)(w)- \theta(x, z)(v)+ \theta(y, z)(u))\\
=&[d(x),y,z]+[x,d(y),z]+[x,y,d(z)]+\lambda[x,y,z]+ D(d(x),y)w+D(x,d(y))w+D(x,y)d_\mathfrak{V}(w)\\
&+\lambda D(x,y)w-\theta(d(x),z)v-\theta(x,d(z))v-\theta(x,z)d_\mathfrak{V}(v)-\lambda\theta(x,z)v+ \theta(d(y),z)u+\theta(y,d(z))u\\
&+\theta(y,z)d_\mathfrak{V}(u)+\lambda\theta(y,z)u\\
=&[d(x),y,z]+ D(d(x),y)w-\theta(d(x),z)v+\theta(y,z)d_\mathfrak{V}(u)\\
&+[x,d(y),z]+D(x,d(y))w-\theta(x,z)d_\mathfrak{V}(v)+ \theta(d(y),z)u\\
&+[x,y,d(z)]+D(x,y)d_\mathfrak{V}(w)-\theta(x,d(z))v+\theta(y,d(z))u\\
&+\lambda([x,y,z]+ D(x,y)w-\theta(x,z)v+\theta(y,z)u)\\
=&[d \oplus d_\mathfrak{V}(x+u),y+v,z+w]_{\ltimes}+[x+u,d \oplus d_\mathfrak{V}(y+v),z+w]_{\ltimes}+[x+u,y+v,d \oplus d_\mathfrak{V}(z+w)]_{\ltimes}\\
&+\lambda[x+u,y+v,z+w]_{\ltimes}.
\end{align*}
Therefore, $(\mathfrak{L} \oplus \mathfrak{V},[\cdot, \cdot, \cdot]_{\ltimes},d \oplus d_\mathfrak{V})$ is a modified $\lambda$-differential Lie triple system.
\end{proof}

Let $(\mathfrak{V};  \theta, d_\mathfrak{V})$  be a representation of a modified $\lambda$-differential Lie triple system $(\mathfrak{L}, [\cdot, \cdot, \cdot ],d)$, and $\mathfrak{V}^*$ be a dual space of $\mathfrak{V}$. We define a bilinear map $\theta^*: \mathfrak{L}\times \mathfrak{L}\rightarrow \mathrm{End}(\mathfrak{V}^*)$ and a linear map $d_\mathfrak{V}^*: \mathfrak{V}^*\rightarrow \mathfrak{V}^*$, respectively by
\begin{align}
&&\langle \theta^*(a,b)u^*,v \rangle=-\langle u^*,\theta(a,b)v \rangle, ~\mathrm{and}~\langle d_\mathfrak{V}^*u^*,v \rangle=\langle u^*,d_\mathfrak{V}(v) \rangle, \label{2.9}
\end{align}
for any $a,b\in \mathfrak{L}, v\in \mathfrak{V}$ and $u^*\in \mathfrak{V}^*.$

Give the switching operator $\tau:\mathfrak{L}\otimes \mathfrak{L}\rightarrow \mathfrak{L}\otimes \mathfrak{L}$ by $\tau(a\otimes b)=\tau(b\otimes a)$, for any $a,b\in \mathfrak{L}.$

\begin{prop}
  With the above notations, $(\mathfrak{V}^*; -\theta^*\tau,-d^*_\mathfrak{V})$ is a representation of modified $\lambda$-differential Lie triple system $(\mathfrak{L}, [\cdot, \cdot, \cdot ],d)$.
  We call it the dual representation of $(\mathfrak{V};  \theta, d_\mathfrak{V})$
\end{prop}
\begin{proof}
Following \cite{Sheng},  we can easily get $(\mathfrak{V}^*; -\theta^*\tau)$ is a representation of the Lie triple system $(\mathfrak{L}, [\cdot, \cdot, \cdot ])$.
Moreover, for any $a,b\in \mathfrak{L},v\in \mathfrak{V}$ and $u^*\in \mathfrak{V}^*,$ by Eqs.  \eqref{2.7} and \eqref{2.9}, we have
\begin{align*}
&\langle -\theta^*\tau(d(a),b)u^*-\theta^*\tau(a,d(b))u^*+\theta^*\tau(a,b)d^*_\mathfrak{V}u^*-\lambda\theta^*\tau(a,b)u^*- d^*_\mathfrak{V}\theta^*\tau(a,b)u^*,v \rangle\\
=&\langle u^*, \theta(b, d(a))v+\theta(d(b),a)v-d_\mathfrak{V}(\theta(b,a)v)+\lambda\theta(b,a)v+\theta(b,a)d_\mathfrak{V}(v) \rangle\\
=&0,
\end{align*}
which implies that $(\mathfrak{V}^*; -\theta^*\tau,-d^*_\mathfrak{V})$ is a representation of $(\mathfrak{L}, [\cdot, \cdot, \cdot ],d)$.

\end{proof}

\begin{exam}
Let  $(\mathfrak{L};    \mathcal{R}, d)$ be an adjoint representation of the modified $\lambda$-differential Lie triple system  $(\mathfrak{L}, [\cdot, \cdot,\cdot],   d)$.
Then, $(\mathfrak{L}^*;    -\mathcal{R}^*\tau, -d^*)$ is a dual adjoint representation of   $(\mathfrak{L}, [\cdot, \cdot,\cdot],   d)$.
\end{exam}

 \section{Cohomology of modified $\lambda$-differential Lie triple systems}\label{sec:cohomologydo}\
 \def\theequation{\arabic{section}.\arabic{equation}}
\setcounter{equation} {0}

In this section, we   study  the cohomology of a modified $\lambda$-differential Lie triple system with
coefficients in its representation.

Let $(\mathfrak{V};  \theta)$  be a representation of a  Lie triple system $(\mathfrak{L}, [\cdot, \cdot, \cdot ])$.
Denote the $(2n+1)-$cochains   of $\mathfrak{L}$ with coefficients in representation $(\mathfrak{V};  \theta)$   by
\begin{align*}
\mathcal{C}_{\mathrm{Lts}}^{2n+1}(\mathfrak{L},\mathfrak{V}):=\big\{f\in \mathrm{Hom}(\mathfrak{L}^{\otimes2n+1},\mathfrak{V}) ~|&~f(a_1,\cdots, a_{2n-2},a,b,c)+f(a_1,\cdots, a_{2n-2},b,a,c)=0,\\
&~\circlearrowright_{a,b,c}f(a_1,\cdots, a_{2n-2},a,b,c)=0  \big\}.
\end{align*}

The Yamaguti coboundary operator $\delta: \mathcal{C}_{\mathrm{Lts}}^{2n-1}(\mathfrak{L},\mathfrak{V})\rightarrow \mathcal{C}_{\mathrm{Lts}}^{2n+1}(\mathfrak{L},\mathfrak{V})$,  for $a_1,\cdots,a_{2n+1}\in \mathfrak{L}$ and $f\in \mathcal{C}_{\mathrm{Lts}}^{2n-1}(\mathfrak{L},\mathfrak{V})$, as
\begin{align*}
&\delta f(a_1,\cdots, a_{2n+1})\\
=&\theta(a_{2n},a_{2n+1})f(a_1,\cdots, a_{2n-1})-\theta(a_{2n-1},a_{2n+1})f(a_1,\cdots, a_{2n-2},a_{2n})\\
&+\sum_{i=1}^n(-1)^{i+n}D(a_{2i-1},a_{2i})f(a_1,\cdots, a_{2i-2},a_{2i+1},\cdots, a_{2n+1})\\
&+\sum_{i=1}^n\sum_{j=2i+1}^{2n+1}(-1)^{i+n+1}f(a_1,\cdots, a_{2i-2},a_{2i+1},\cdots, [a_{2i-1},a_{2i},a_{j}],\cdots,a_{2n+1}).
\end{align*}
So $\delta\circ\delta=0.$
One can refer to \cite{Jacobson1,Yamaguti,Kubo,Zhang} for more information about  Lie triple systems and cohomology theory.\\

Next, we introduce a cohomology of a modified $\lambda$-differential Lie triple system   with coefficients in a representation.

We first give the following lemma.
\begin{lemma}\label{lemma:cochain map}
Let $(\mathfrak{V};  \theta, d_\mathfrak{V})$  be a representation of a modified $\lambda$-differential Lie triple system $(\mathfrak{L}, [\cdot, \cdot, \cdot ],d)$.
For any $n\geq 1$, we define a linear map $\Phi: \mathcal{C}_{\mathrm{Lts}}^{2n-1}(\mathfrak{L},\mathfrak{V})\rightarrow \mathcal{C}_{\mathrm{Lts}}^{2n-1}(\mathfrak{L},\mathfrak{V})$ by
\begin{align*}
\Phi f(a_1,\cdots, a_{2n-1})=&\sum_{i=1}^{2n-1}f(a_1,\cdots,d(a_{i}),\cdots,a_{2n-1})\\
&+(n-1)\lambda f(a_1,\cdots,a_{2n-1})-d_\mathfrak{V}(f(a_1,\cdots, a_{2n-1})),
\end{align*}
for  any $f\in \mathcal{C}_{\mathrm{Lts}}^{2n-1}(\mathfrak{L},\mathfrak{V})$ and $a_1,\cdots, a_{2n-1}\in \mathfrak{L}.$
Then, $\Phi$ is a cochain map, that is, the following diagram is commutative:
$$\aligned
\xymatrix{
  \mathcal{C}^{2n-1}_{\mathrm{Lts}}(\mathfrak{L},\mathfrak{V})\ar[r]^-{\delta}\ar[d]^-{\Phi}& \mathcal{C}^{2n+1}_{\mathrm{Lts}}(\mathfrak{L},\mathfrak{V})\ar[d]^{\Phi}\\
  \mathcal{C}^{2n-1}_{{\mathrm{Lts}}}(\mathfrak{L},\mathfrak{V})\ar[r]^-{\delta}& \mathcal{C}^{2n+1}_{{\mathrm{Lts}}}(\mathfrak{L},\mathfrak{V}).}
 \endaligned$$
\end{lemma}

\begin{proof}
We only prove the case of $n=1$. In the other cases, the proof is similar.
For all $f\in \mathcal{C}_{\mathrm{Lts}}^{1}(\mathfrak{L},\mathfrak{V})$ and $a_1,a_2, a_{3}\in \mathfrak{L},$  by Eqs. \eqref{2.6}, \eqref{2.7} and \eqref{2.8},  we have
\begin{align*}
&\Phi\circ\delta f(a_1,a_2,a_3)\\
=&\delta f(d(a_1),a_{2},a_{3})+\delta f(a_1,d(a_{2}),a_{3})+\delta f(a_1,a_{2},d(a_{3}))+\lambda \delta f(a_1,a_2,a_{3})-d_\mathfrak{V}(\delta f(a_1,a_2, a_{3}))\\
=&\theta(a_2,a_3)f(d(a_1))-\theta(d(a_1),a_3)f(a_2)+D(d(a_1),a_2)f(a_3)-f([d(a_1),a_2,a_3])\\
&+\theta(d(a_{2}),a_3)f(a_1)-\theta(a_1,a_3)f(d(a_{2}))+D(a_1,d(a_{2}))f(a_3)-f([a_1,d(a_{2}),a_3])\\
&+\theta(a_2,d(a_{3}))f(a_1)-\theta(a_1,d(a_{3}))f(a_2)+D(a_1,a_2)f(d(a_{3}))-f([a_1,a_2,d(a_{3})])\\
&+\lambda\theta(a_2,a_3)f(a_1)-\lambda\theta(a_1,a_3)f(a_2)+\lambda D(a_1,a_2)f(a_3)-\lambda f([a_1,a_2,a_3])\\
&-d_\mathfrak{V}(\theta(a_2,a_3)f(a_1))+d_\mathfrak{V}(\theta(a_1,a_3)f(a_2))-d_\mathfrak{V}(D(a_1,a_2)f(a_3))+d_\mathfrak{V}(f([a_1,a_2,a_3]))\\
=&\theta(a_2,a_3)f(d(a_1))-\theta(a_2,a_3)d_\mathfrak{V}(f(a_1))-\theta(a_1,a_3)f(d(a_{2}))+\theta(a_1,a_3)d_\mathfrak{V}(f(a_2))\\
&+D(a_1,a_2)f(d(a_{3}))-D(a_1,a_2)d_\mathfrak{V} f(a_3)- f(d([a_1,a_2,a_3]))+d_\mathfrak{V}(f([a_1,a_2,a_3]))\\
=&\theta(a_2,a_3)\Phi f(a_1)-\theta(a_1,a_3)\Phi f(a_2)+D(a_1,a_2)\Phi f(a_3)-\Phi f([a_1,a_2,a_3])\\
=&\delta\circ\Phi f(a_1,a_2,a_3).
\end{align*}
Therefore, $\Phi\circ\delta=\delta\circ\Phi.$
\end{proof}

Define the set of $(2n-1)$-cochains by
\begin{equation*}\label{eq:dac}
\mathcal{C}_{\mathrm{mDLts^\lambda}}^{2n-1}(\mathfrak{L},\mathfrak{V}):=
\begin{cases}
\mathcal{C}^{2n-1}_{\mathrm{Lts}}(\mathfrak{L},\mathfrak{V})\oplus \mathcal{C}^{2n-3}_{\mathrm{Lts}}(\mathfrak{L},\mathfrak{V}),&n\geq 2,\\
\mathcal{C}^{1}_{\mathrm{Lts}}(\mathfrak{L},\mathfrak{V})=\mathrm{Hom}(\mathfrak{L},\mathfrak{V}),&n=1.
\end{cases}
\end{equation*}

For $n \geq 2$,   we define   linear map  $\partial:\mathcal{C}_{\mathrm{mDLts^\lambda}}^{2n-1}(\mathfrak{L},\mathfrak{V})\rightarrow \mathcal{C}_{\mathrm{mDLts^\lambda}}^{2n+1}(\mathfrak{L},\mathfrak{V})$  by

\begin{align*}
\partial(f,g)=(\delta f, \delta g+(-1)^n \Phi f), \forall (f,g)\in \mathcal{C}_{\mathrm{mDLts^\lambda}}^{2n-1}(\mathfrak{L},\mathfrak{V}).
\end{align*}

When $n=1$,   define  the linear map  $\partial:\mathcal{C}_{\mathrm{mDLts^\lambda}}^{1}(\mathfrak{L},\mathfrak{V})\rightarrow \mathcal{C}_{\mathrm{mDLts^\lambda}}^{3}(\mathfrak{L},\mathfrak{V})$  by
\begin{align*}
\partial(f)=(\delta f, -\Phi f), \forall f\in \mathcal{C}_{\mathrm{mDLts^\lambda}}^{1}(\mathfrak{L},\mathfrak{V}).
\end{align*}

\begin{theorem}\label{thm: cochain complex for differential algebras}
The pair $(\mathcal{C}_{\mathrm{mDLts^\lambda}}^*(\mathfrak{L}, \mathfrak{V}),\partial)$ is a cochain complex. So $\partial\circ\partial=0.$
\end{theorem}
\begin{proof}
For any $f\in \mathcal{C}_{\mathrm{mDLts^\lambda}}^{1}(\mathfrak{L},\mathfrak{V})$, by Lemma \ref{lemma:cochain map}, we have
$$\partial\circ\partial(f) =\partial(\delta f, -\Phi f)=(\delta (\delta f), \delta (-\Phi f)+\Phi( \delta f))=0.$$
Given any $(f,g)\in \mathcal{C}_{\mathrm{mDLts^\lambda}}^{2n-1}(\mathfrak{L},\mathfrak{V})$ with $n\geq 2$,  we have
$$\partial\circ\partial(f,g)=\partial(\delta f, \delta g+(-1)^n \Phi f)=(\delta (\delta f), \delta (\delta g+(-1)^n \Phi f)+(-1)^{n+1} \Phi (\delta f))=0.$$
Therefore,   $(\mathcal{C}_{\mathrm{mDLts^\lambda}}^*(\mathfrak{L}, \mathfrak{V}),\partial)$ is a cochain complex.
\end{proof}

\begin{defn}
  The cohomology of the cochain complex $(\mathcal{C}_{\mathrm{mDLts^\lambda}}^*(\mathfrak{L},\mathfrak{ V}),\partial)$, denoted by $\mathcal{H}_{\mathrm{mDLts^\lambda}}^*(\mathfrak{L},$
  $ \mathfrak{V})$, is called the   cohomology of the  modified $\lambda$-differential Lie triple system $(\mathfrak{L}, [\cdot, \cdot, \cdot ],d)$  with coefficients in the representation    $(\mathfrak{V};  \theta, d_\mathfrak{V})$.
\end{defn}

For $n\geq 1$,  we denote the set of $(2n+1)$-cocycles by
$\mathcal{Z}^{2n+1}_{\mathrm{mDLts^\lambda}}(\mathfrak{L},\mathfrak{ V})=\big\{(f,g)\in\mathcal{C}^{2n+1}_{\mathrm{mDLts^\lambda}}(\mathfrak{L},\mathfrak{ V})~|~$
$\partial (f,g)=0 \big\}$, the set of $(2n+1)$-coboundaries by $\mathcal{B}^{2n+1}_{\mathrm{mDLts^\lambda}}(\mathfrak{L},\mathfrak{ V})=\{\partial (f,g)~|~(f,g)\in\mathcal{C}^{2n-1}_{\mathrm{mDLts^\lambda}}(\mathfrak{L},\mathfrak{ V})\}$ and the $(2n+1)$-th cohomology group of the  modified $\lambda$-differential Lie triple system  $(\mathfrak{L}, [\cdot, \cdot, \cdot ],d)$  with coefficients in the representation    $(\mathfrak{V};  \theta, d_\mathfrak{V})$ by $\mathcal{H}^{2n+1}_{\mathrm{mDLts^\lambda}}(\mathfrak{L},\mathfrak{ V})=\mathcal{Z}^{2n+1}_{\mathrm{mDLts^\lambda}}(\mathfrak{L},\mathfrak{ V})/\mathcal{B}^{2n+1}_{\mathrm{mDLts^\lambda}}(\mathfrak{L},\mathfrak{ V})$.

To end this section, we compute  1-cocycles and 3-cocycles of the  modified $\lambda$-differential Lie triple system  $(\mathfrak{L}, [\cdot, \cdot, \cdot ],d)$  with coefficients in the representation    $(\mathfrak{V};  \theta, d_\mathfrak{V})$.

It is obvious that for all $f\in \mathcal{C}_{\mathrm{mDLts^\lambda}}^{1}(\mathfrak{L},\mathfrak{V})$, $f$ is a 1-cocycle if and only if $\partial(f)=(\delta f, -\Phi f)=0,$ that is,
$$\theta(a_2,a_3)f(a_1)-\theta(a_1,a_3)f(a_2)+D(a_1,a_2)f(a_3)-f([a_1,a_2,a_3])=0$$
and
$$d_\mathfrak{V}(f(a_1))-f(d(a_1))=0.$$

For all $(f,g)\in \mathcal{C}_{\mathrm{mDLts^\lambda}}^{3}(\mathfrak{L},\mathfrak{V})$, $(f,g)$ is a 3-cocycle if and only if $\partial(f,g)=(\delta f, \delta g+\Phi f)=0,$ that is,
\begin{align*}
&\theta(a_4,a_5)f(a_1,a_2,a_3)-\theta(a_3,a_5)f(a_1,a_2,a_4)-D(a_1,a_2)f(a_3,a_4,a_5)+D(a_3,a_4)f(a_1,a_2,a_5)\\
&+f([a_1,a_2,a_3],a_4,a_5)+f(a_3,[a_1,a_2,a_4],a_5)+f(a_3,a_4,[a_1,a_2,a_5])-f(a_1,a_2,[a_3,a_4,a_5])=0
\end{align*}
and
\begin{align*}
&\theta(a_2,a_3)g(a_1)-\theta(a_1,a_3)g(a_2)+D(a_1,a_2)g(a_3)-g([a_1,a_2,a_3])\\
&+f(d(a_1),a_2,a_3)+f(a_1,d(a_2),a_3)+f(a_1,a_2,d(a_3))+\lambda f(a_1,a_2,a_3)-d_\mathfrak{V}(f(a_1,a_2,a_3))=0.
\end{align*}

\section{Deformations of modified $\lambda$-differential Lie triple systems}\label{sec:def}
\def\theequation{\arabic{section}.\arabic{equation}}
\setcounter{equation} {0}

 In this section, we  introduce formal deformations  of  the modified $\lambda$-differential Lie triple system. Furthermore, we show that if the third cohomology group $\mathcal{H}_{\mathrm{mDLts^\lambda}}^3(\mathfrak{L}, \mathfrak{L})=0$, then the   modified $\lambda$-differential Lie triple system $(\mathfrak{L}, [\cdot, \cdot, \cdot ],d)$ is rigid.

 Let $(\mathfrak{L}, [\cdot, \cdot, \cdot ],d)$ be a   modified $\lambda$-differential Lie triple system. Denote by $\nu$ the multiplication of $\mathfrak{L}$, i.e., $\nu=[\cdot, \cdot, \cdot ]$.
Consider the 1-parameterized family
$$\nu_t=\sum_{i=0}^{\infty} \nu_i t^i, \, \, \nu_i\in \mathcal{C}^3_{\mathrm{Lts}}(\mathfrak{L}, \mathfrak{L}),\quad
 d_t=\sum_{i=0}^{\infty} d_i t^i, \, \, d_i\in \mathcal{C}^1_{\mathrm{Lts}}(\mathfrak{L}, \mathfrak{L}).$$

\begin{defn}
A  1-parameter formal deformation of the modified $\lambda$-differential Lie triple system $(\mathfrak{L}, \nu,d)$  is a pair $(\nu_t, d_t)$ which endows the $\mathbb{K}[[t]]$-module $(\mathfrak{L}[[t]], \nu_t, d_t)$ with the modified $\lambda$-differential Lie triple system   {over $\mathbb{K}[[t]]$} such that $(\nu_0, d_0)=(\nu, d)$.
\end{defn}

Obviously, $(\mathfrak{L}[[t]], \nu_t=\nu, d_t=d)$ is a 1-parameter formal deformation  of  $(\mathfrak{L}, \nu,d)$.

 The pair $(\nu_t, d_t)$ generates a 1-parameter formal deformation  of the modified $\lambda$-differential Lie triple system $(\mathfrak{L}, [\cdot, \cdot, \cdot ],d)$  if and only if for all $a, b, c,x,y\in \mathfrak{L}$, the following equations hold:
\begin{align}
&\nu_t(a,b,c)+\nu_t(b,a,c)=0,\label{4.1}\\
&\nu_t(a,b,c))+\nu_t(b,c,a)+\nu_t(c,a,b)=0,\label{4.2}\\
&\nu_t(x,y,\nu_t(a,b,c))=\nu_t(\nu_t(x,y,a),b,c)+\nu_t(a,\nu_t(x,y,b),c)+\nu_t(a,b,\nu_t(x,y,c)),\label{4.3}\\
 &d_t(\nu_t(a,b,c))=\nu_t(d_t(a),b,c)+\nu_t(a, d_t(b),c)+\nu_t(a, b, d_t(c))+\lambda \nu_t(a,b,c).\label{4.4}
\end{align}
Comparing the coefficients of $t^n$ on both sides of the above equations, Eqs. \eqref{4.1}-\eqref{4.4} are equivalent to the following equations:
\begin{align}
&\nu_n(a,b,c)+\nu_n(b,a,c)=0,\label{4.5}\\
&\nu_n(a,b,c))+\nu_n(b,c,a)+\nu_n(c,a,b)=0,\label{4.6}\\
&\sum_{i+j=n}\nu_i(x,y,\nu_j(a,b,c))=\sum_{i+j=n}(\nu_i(\nu_j(x,y,a),b,c)+\nu_i(a,\nu_j(x,y,b),c)+\nu_i(a,b,\nu_j(x,y,c))),\label{4.7}\\
 &\sum_{i+j=n}d_i(\nu_j(a,b,c))=\sum_{i+j=n}(\nu_i(d_j(a),b,c)+\nu_i(a, d_j(b),c)+\nu_i(a, b, d_j(c)))+\lambda \nu_n(a,b,c).\label{4.8}
\end{align}

 \begin{prop}\label{prop:fddco}
Let $(\mathfrak{L}[[t]], \nu_t, d_t)$ be a $1$-parameter formal deformation of the modified $\lambda$-differential Lie triple system $(\mathfrak{L}, \nu,d)$. Then $(\nu_1, d_1)$ is a 3-cocycle of $(\mathfrak{L}, \nu,d)$ with the coefficient  in the adjoint representation $(\mathfrak{L};\mathcal{R},d)$.	
\end{prop}
\begin{proof}
  For $n =1$, Eq.~\eqref{4.7} is equivalent to
\begin{align*}
&\nu_1(x,y,[a,b,c])+[x,y,\nu_1(a,b,c)]\\
 =& \nu_1([x,y,a],b,c)+[\nu_1(x,y,a),b,c]+\nu_1(a,[x,y,b],c)+[a,\nu_1(x,y,b),c]+\nu_1(a,b,[x,y,c])\\
 &+[a,b,\nu_1(x,y,c)],
 \end{align*}
i.e., $\delta \nu_1=0.$ In addition, for $n =1$, Eq.~\eqref{4.8} is equivalent to
\begin{align*}
&d_1([a,b,c])+ d(\nu_1(a,b,c))\\
 =& [d_1(a),b,c]+\nu_1(d(a),b,c)+[a,d_1(b),c]+\nu_1(a,d(b),c)+[a,b,d_1(c)]+\nu_1(a, b,d(c))\\
 &+\lambda\nu_1(a,b,c),
 \end{align*}
that is, $\delta d_1+ \Phi\nu_1=0.$
In other words, Eqs.~\eqref{4.7} and \eqref{4.8} are equivalent to $\partial(\nu_1,d_1)=(\delta \nu_1,\delta d_1+ \Phi\nu_1)=0.$  Therefore, 	$(\nu_1, d_1)$ is a 3-cocycle of $(\mathfrak{L}, \nu,d)$ with the coefficient  in the adjoint representation $(\mathfrak{L};\mathcal{R},d)$.	
\end{proof}

If $\nu_t=\nu$ in the above $1$-parameter formal deformation of the modified $\lambda$-differential Lie triple system $(\mathfrak{L}, \nu,d)$, we obtain a $1$-parameter formal deformation of the modified $\lambda$-differential operator $d$. So we have

 \begin{coro}\label{coro:fdo}
Let $ d_t $ be a $1$-parameter formal deformation of the modified $\lambda$-differential operator $d$.  Then $d_1$ is a 1-cocycle of the modified $\lambda$-differential operator $d$ with coefficients  in the adjoint representation $(\mathfrak{L};\mathcal{R},d)$.	
\end{coro}
\begin{proof}
When $n =1$, by $\nu_1=0$ and Eq. \eqref{4.8},  we have   $d_1\in \mathrm{Der}(\mathfrak{L})$. That is, Eq. \eqref{4.8} is equivalent to $\delta d_1=0$, which implies that  $d_1$ is a 1-cocycle of the modified $\lambda$-differential operator $d$ with coefficients in the adjoint representation $(\mathfrak{L};\mathcal{R},d)$.	
\end{proof}

\begin{defn}
The $3$-cocycle $(\nu_1,d_1)$ is called the   infinitesimal  of the $1$-parameter formal deformation $(\mathfrak{L}[[t]],\nu_t,d_t)$ of $(\mathfrak{L}, \nu,d)$.
\end{defn}

\begin{defn}
Let $(\mathfrak{L}[[t]],\nu_t,d_t)$ and $(\mathfrak{L}[[t]],\nu'_t,d'_t)$ be two $1$-parameter formal deformations of $(\mathfrak{L},\nu,d)$. A
  formal isomorphism  from $(\mathfrak{L}[[t]], \nu'_t, d'_t)$ to $(\mathfrak{L}[[t]],\nu_t,d_t)$ is a power series $\varphi_t=\mathrm{id}_\mathfrak{L}+\sum_{i=1}^{\infty}\varphi_it^i:(\mathfrak{L}[[t]], \nu'_t, d'_t) \rightarrow (\mathfrak{L}[[t]],\nu_t,d_t)$, where $\varphi_i\in \mathrm{End}(\mathfrak{L})$, such that
\begin{align}
\varphi_t\circ \nu'_t=& \nu_t\circ(\varphi_t\times\varphi_t\times\varphi_t),\label{4.9}\\
\varphi_t\circ d'_t=&d_t\circ\varphi_t.\label{4.10}
\end{align}

Two $1$-parameter formal deformations $(\mathfrak{L}[[t]],\nu_t,d_t)$ and $(\mathfrak{L}[[t]], \nu'_t, d'_t)$ are said to be   equivalent  if  there exists a formal isomorphism $\varphi_t:(\mathfrak{L}[[t]], \nu'_t, d'_t) \rightarrow (\mathfrak{L}[[t]],\nu_t,d_t)$.
\end{defn}

\begin{theorem}
The infinitesimals of two equivalent $1$-parameter formal deformations of $(\mathfrak{L}, \nu,d)$ are in the same cohomology class $\mathcal{H}_{\mathrm{mDLts^{\lambda}}}^3(\mathfrak{L},\mathfrak{L})$.
\end{theorem}
\begin{proof}
Let $\varphi_t:(\mathfrak{L}[[t]], \nu'_t, d'_t)\rightarrow (\mathfrak{L}[[t]],\nu_t,d_t)$ be a formal isomorphism. For all $a,b,c\in \mathfrak{L}$, we have
\begin{eqnarray*}
\varphi_t\circ \nu'_t(a,b,c)&=& \nu_t(\varphi_t(a),\varphi_t(b),\varphi_t(c)),\\
\varphi_t\circ d'_t(a)&=& d_t\circ\varphi_t (a).
\end{eqnarray*}
Comparing the coefficients of $t$ on both sides of the above equations, we can get
\begin{eqnarray*}
\nu'_1(a,b,c)-\nu_1(a,b,c)&=&[\varphi_1(a), b,c]+[a, \varphi_1(b),c]++[a, b,\varphi_1(c)]-\varphi_1([a, b,c]),\\
 d'_1(a)-d_1(a)&=&d(\varphi_1(a))-\varphi_1(d(a)).
\end{eqnarray*}
Therefore, we have $$(\nu'_1,d'_1)-(\nu_1,d_1)=(\delta\varphi_1,-\Phi\varphi_1)=\partial(\varphi_1)\in \mathcal{C}_{\mathrm{mDLts^{\lambda}}}^3(\mathfrak{L},\mathfrak{L}),$$ which implies that $[(\nu'_1,d'_1)]=[(\nu_1,d_1)]$ in $\mathcal{H}_{\mathrm{mDLts^{\lambda}}}^3(\mathfrak{L},\mathfrak{L})$.
\end{proof}

\begin{defn}
(i) A $1$-parameter formal deformation $(\mathfrak{L}[[t]],\nu_t,d_t)$ of $(\mathfrak{L}, \nu,d)$ is said to be   trivial  if it is equivalent to the  deformation $(\mathfrak{L}[[t]],  \nu,d)$.
\\
(ii) A  modified $\lambda$-differential Lie triple system $(\mathfrak{L}, \nu,d)$ is said to be   rigid  if every $1$-parameter formal deformation  is trivial.
\end{defn}

\begin{theorem}
If $\mathcal{H}_{\mathrm{mDLts^{\lambda}}}^3(\mathfrak{L},\mathfrak{L})=0$, then  the modified $\lambda$-differential Lie triple system $(\mathfrak{L}, \nu,d)$ is rigid.
\end{theorem}
\begin{proof}
Let $(\mathfrak{L}[[t]],\nu_t,d_t)$ be a $1$-parameter formal deformation of $(\mathfrak{L}, \nu,d)$. By Proposition ~\ref{prop:fddco},   $(\nu_1,d_1)$ is a 3-cocycle. By $\mathcal{H}_{\mathrm{mDLts}^{\lambda}}^3(\mathfrak{L},\mathfrak{L})=0$, there exists a 1-cochain $\varphi_1\in  \mathcal{C}^1_{\mathrm{mDLts^\lambda}}(\mathfrak{L},\mathfrak{L})$ such that
\begin{eqnarray}
(\nu_1,d_1)=\partial(\varphi_1).\label{4.11}
\end{eqnarray}
Then setting $\varphi_t=\mathrm{id}_\mathfrak{L}+\varphi_1 t$, we have a deformation $(\mathfrak{L}[[t]], \nu'_t, d'_t)$, where
\begin{eqnarray*}
\nu'_t &=&\varphi_t^{-1}\circ\nu_t\circ(\varphi_t\otimes\varphi_t\otimes\varphi_t),\\
d'_t &=&\varphi_t^{-1}\circ d_t\circ\varphi_t.
\end{eqnarray*}
Thus, $(\mathfrak{L}[[t]],\nu'_t, d'_t)$ is equivalent to $(\mathfrak{L}[[t]],\nu_t,d_t)$. Moreover, we have
\begin{align*}
\nu'_t=&(\mathrm{id}_L-\varphi_1t+\varphi_1^2t^{2}+\cdots+(-1)^i\varphi_1^it^{i}+\cdots)\circ\nu_t\circ\big((\mathrm{id}_L+\varphi_1 t)\otimes(\mathrm{id}_L+\varphi_1 t)\otimes(\mathrm{id}_L+\varphi_1 t)\big),\\
d'_t=&(\mathrm{id}_L-\varphi_1t+\varphi_1^2t^{2}+\cdots+(-1)^i\varphi_1^it^{i}+\cdots)\circ d_t\circ(\mathrm{id}_L+\varphi_1 t).
\end{align*}
By Eq.~\eqref{4.11}, we have
\begin{eqnarray*}
\nu'_t&=&\nu+ \nu'_{2} t^{2}+\cdots,\\
d'_t&=&d + d'_{2} t^{2}+\cdots.
\end{eqnarray*}
Then by repeating the argument, we can show that $(\mathfrak{L}[[t]],\nu_t,d_t)$ is equivalent to $(\mathfrak{L}[[t]],\nu,d)$. Therefore, $(\mathfrak{L}, \nu,d)$ is rigid.
\end{proof}

\section{Abelian extensions of modified $\lambda$-differential Lie triple systems} \label{sec:ext}
\def\theequation{\arabic{section}.\arabic{equation}}
\setcounter{equation} {0}
In this section, we study abelian extensions of modified $\lambda$-differential Lie triple systems and prove that they are classified by the third cohomology.

\begin{defn}
Let $(\mathfrak{L}, [\cdot,\cdot,\cdot], d)$ and $(\mathfrak{V},[\cdot,\cdot,\cdot]_\mathfrak{V}, d_\mathfrak{V})$ be two modified $\lambda$-differential Lie triple systems.
An abelian extension of $(\mathfrak{L}, [\cdot,\cdot,\cdot], d)$ by  $(\mathfrak{V},[\cdot,\cdot,\cdot]_\mathfrak{V}, d_\mathfrak{V})$
 is a short exact sequence of homomorphisms of modified $\lambda$-differential Lie triple systems
$$\begin{CD}
0@>>> {\mathfrak{V}} @>i >> \hat{\mathfrak{L}} @>p >> \mathfrak{L} @>>>0\\
@. @V {d_\mathfrak{V}} VV @V \hat{d} VV @V d VV @.\\
0@>>> {\mathfrak{V}} @>i >> \hat{\mathfrak{L}} @>p >> \mathfrak{L} @>>>0
\end{CD}$$
such that $[u, v, \cdot]_{\hat{\mathfrak{L}}}=[u, \cdot, v]_{\hat{\mathfrak{L}}}=[\cdot, u, v]_{\hat{\mathfrak{L}}}=0$, for all $u,v\in \mathfrak{V}$,  i.e., $\mathfrak{V}$ is an abelian ideal of $\hat{\mathfrak{L}}.$
\end{defn}

\begin{defn}
Let $(\hat{\mathfrak{L}}_1, [\cdot,\cdot,\cdot]_{\hat {\mathfrak{L}}_1}, \hat {d_1})$ and $(\hat{\mathfrak{L}}_2, [\cdot,\cdot,\cdot]_{\hat {\mathfrak{L}}_2}, \hat {d_2})$ be two abelian extensions of $(\mathfrak{L}, [\cdot,\cdot,\cdot], d)$ by $(\mathfrak{V},[\cdot,\cdot,\cdot]_\mathfrak{V}, d_\mathfrak{V})$. They are said to be  equivalent if  there is an isomorphism of modified $\lambda$-differential Lie triple systems $\zeta:(\hat{\mathfrak{L}}_1, [\cdot,\cdot,\cdot]_{\hat {\mathfrak{L}}_1}, \hat {d_1})\rightarrow (\hat{\mathfrak{L}}_2, [\cdot,\cdot,\cdot]_{\hat {\mathfrak{L}}_2}, \hat {d_2})$
such that the following diagram is  commutative:
$$\begin{CD}
0@>>> {(\mathfrak{V},d_\mathfrak{V})} @>i >> (\hat{\mathfrak{L}}_1,   \hat {d_1}) @>p >> (\mathfrak{L},d) @>>>0\\
@. @| @V \zeta VV @| @.\\
0@>>> {(\mathfrak{V},d_\mathfrak{V})} @>i >> (\hat{\mathfrak{L}}_2,   \hat {d_2}) @>p >> (\mathfrak{L},d) @>>>0.
\end{CD}$$
\end{defn}

A   section  of an abelian extension $(\hat{\mathfrak{L}}, [\cdot,\cdot,\cdot]_{\hat {\mathfrak{L}}}, \hat {d})$ of $(\mathfrak{L}, [\cdot,\cdot,\cdot], d)$ by  $(\mathfrak{V},[\cdot,\cdot,\cdot]_\mathfrak{V}, d_\mathfrak{V})$ is a linear map $\sigma:\mathfrak{L}\rightarrow \hat{\mathfrak{L}}$ such that $p\circ \sigma=\mathrm{id}_\mathfrak{L}$.

Now for an abelian extension $(\hat{\mathfrak{L}}, [\cdot,\cdot,\cdot]_{\hat {\mathfrak{L}}}, \hat {d})$ of $(\mathfrak{L}, [\cdot,\cdot,\cdot], d)$ by  $(\mathfrak{V},[\cdot,\cdot,\cdot]_\mathfrak{V}, d_\mathfrak{V})$ with a section $\sigma:\mathfrak{L}\rightarrow\hat{\mathfrak{L}}$, we define linear map $\vartheta: \mathfrak{L}\times \mathfrak{L}\rightarrow \mathrm{End}(\mathfrak{V})$  by
$$\vartheta(a,b)u:=[u,\sigma(a),\sigma(b)]_{\hat{\mathfrak{L}}}, \quad \forall a,b\in \mathfrak{L}, u\in \mathfrak{V}.$$
In particular, $\mathcal{D}(a,b)u=[\sigma(a),\sigma(b),u]_{\hat{\mathfrak{L}}}=\vartheta(b,a)u-\vartheta(a,b)u.$
\begin{prop}
  With the above notations, $(\mathfrak{V},\vartheta, d_\mathfrak{V})$ is a representation over the modified $\lambda$-differential Lie triple systems  $(\mathfrak{L}, [\cdot,\cdot,\cdot], d)$.
\end{prop}
\begin{proof}
First, for any  $x,y,a,b\in \mathfrak{L}$ and $u\in \mathfrak{V}$, From $\sigma([y,a,b])-[\sigma(y),\sigma(a),\sigma(b)]_{\hat{\mathfrak{L}}}\in \mathfrak{V}\cong \mathrm{ker}(p)$, we can get
 $[u,\sigma(x),\sigma([y,a,b])]_{\hat{\mathfrak{L}}}=[u,\sigma(x),[\sigma(y),\sigma(a),\sigma(b)]_{\hat{\mathfrak{L}}}]_{\hat{\mathfrak{L}}}$.
 Furthermore, by Eqs. \eqref{2.1} and  \eqref{2.3},   we obtain
\begin{align*}
& \vartheta(a,b)\vartheta(x,y)u-\vartheta(y,b)\vartheta(x,a)u-\vartheta(x, [y,a,b])u+\mathcal{D}(y,a)\vartheta(x,b)u\\
=&[[u,\sigma(x),\sigma(y)]_{\hat{\mathfrak{L}}},\sigma(a),\sigma(b)]_{\hat{\mathfrak{L}}}-[[u,\sigma(x),\sigma(a)]_{\hat{\mathfrak{L}}},\sigma(y),\sigma(b)]_{\hat{\mathfrak{L}}}-[u,\sigma(x),\sigma([y,a,b])]_{\hat{\mathfrak{L}}}\\
&+[\sigma(y),\sigma(a),[u,\sigma(x),\sigma(b)]_{\hat{\mathfrak{L}}}]_{\hat{\mathfrak{L}}}\\
=&[[u,\sigma(x),\sigma(y)]_{\hat{\mathfrak{L}}},\sigma(a),\sigma(b)]_{\hat{\mathfrak{L}}}+[\sigma(y),[u,\sigma(x),\sigma(a)]_{\hat{\mathfrak{L}}},\sigma(b)]_{\hat{\mathfrak{L}}}-[u,\sigma(x),[\sigma(y),\sigma(a),\sigma(b)]_{\hat{\mathfrak{L}}})]_{\hat{\mathfrak{L}}}\\
&+[\sigma(y),\sigma(a),[u,\sigma(x),\sigma(b)]_{\hat{\mathfrak{L}}}]_{\hat{\mathfrak{L}}}\\
=&0,\\
& \vartheta(a,b)\mathcal{D}(x,y)u-\mathcal{D}(x,y)\vartheta(a,b)u+\vartheta([x,y,a],b)u+\vartheta(a,[x,y,b])u\\
=&[[\sigma(x),\sigma(y),u]_{\hat{\mathfrak{L}}},\sigma(a),\sigma(b)]_{\hat{\mathfrak{L}}}-[\sigma(x),\sigma(y),[u,\sigma(a),\sigma(b)]_{\hat{\mathfrak{L}}}]_{\hat{\mathfrak{L}}}+[u,\sigma([x,y,a]),\sigma(b)]_{\hat{\mathfrak{L}}}\\
&+[u,\sigma(a),\sigma([x,y,b])]_{\hat{\mathfrak{L}}}\\
=&[[\sigma(x),\sigma(y),u]_{\hat{\mathfrak{L}}},\sigma(a),\sigma(b)]_{\hat{\mathfrak{L}}}-[\sigma(x),\sigma(y),[u,\sigma(a),\sigma(b)]_{\hat{\mathfrak{L}}}]_{\hat{\mathfrak{L}}}+[u,[\sigma(x),\sigma(y),\sigma(a)]_{\hat{\mathfrak{L}}},\sigma(b)]_{\hat{\mathfrak{L}}}\\
&+[u,\sigma(a),[\sigma(x),\sigma(y),\sigma(b)]_{\hat{\mathfrak{L}}}]_{\hat{\mathfrak{L}}}\\
=&0.
\end{align*}

In addition,  $\hat{d}(\sigma(x))-\sigma(d(x))\in \mathfrak{V}\cong \mathrm{ker}(p),$ means that  $[u,\sigma(d(x)),\sigma(y)]_{\hat{\mathfrak{L}}}= [u,\hat{d}(\sigma(x)),\sigma(y)]_{\hat{\mathfrak{L}}}$. Thus, we have
\begin{align*}
&d_\mathfrak{V}(\vartheta(x,y)u)=d_\mathfrak{V}([u,\sigma(x),\sigma(y)]_{\hat{\mathfrak{L}}})\\
=&[d_\mathfrak{V}(u),\sigma(x),\sigma(y)]_{\hat{\mathfrak{L}}}+[u,\hat{d}(\sigma(x)),\sigma(y)]_{\hat{\mathfrak{L}}}+[u,\sigma(x),\hat{d}(\sigma(y))]_{\hat{\mathfrak{L}}}+\lambda[u,\sigma(x),\sigma(y)]_{\hat{\mathfrak{L}}}\\
=&[d_\mathfrak{V}(u),\sigma(x),\sigma(y)]_{\hat{\mathfrak{L}}}+[u,\sigma(d(x)),\sigma(y)]_{\hat{\mathfrak{L}}}+[u,\sigma(x),\sigma(d(y)))]_{\hat{\mathfrak{L}}}+\lambda[u,\sigma(x),\sigma(y)]_{\hat{\mathfrak{L}}}\\
=&\vartheta(x,y)d_\mathfrak{V}(u)+\vartheta(d(x),y)u+\vartheta(x,d(y))u+\lambda\vartheta(x,y)u,
\end{align*}
Hence, $(\mathfrak{V},\vartheta, d_\mathfrak{V})$ is a representation over  $(\mathfrak{L}, [\cdot,\cdot,\cdot], d)$.
\end{proof}

We  further  define linear maps $\varsigma:\mathfrak{L}\times \mathfrak{L}\times \mathfrak{L}\rightarrow \mathfrak{V}$ and $\varpi:\mathfrak{L}\rightarrow \mathfrak{V}$ respectively by
\begin{align*}
\varsigma(a,b,c)&=[\sigma(a), \sigma(b),\sigma(c)]_{\hat{\mathfrak{L}}}-\sigma([a, b,c]),\\
\varpi(a)&=\hat{d}(\sigma(a))-\sigma(d(a)),\quad\forall a,b,c\in \mathfrak{L}.
\end{align*}
We transfer the modified $\lambda$-differential Lie triple system structure on $\hat{\mathfrak{L}}$ to $\mathfrak{L}\oplus \mathfrak{V}$ by endowing $\mathfrak{L}\oplus \mathfrak{V}$ with a multiplication $[\cdot, \cdot,\cdot]_\varsigma$ and the modified $\lambda$-differential
operator  $d_\varpi$  defined by
\begin{align}
[a+u, b+v,c+w]_\varsigma&=[a, b, c]+\vartheta(b,c)u-\vartheta(a,c)v+\mathcal{D}(a,b)w+\varsigma(a,b,c), \label{5.1}\\
d_\varpi(a+u)&=d(a)+\varpi(a)+d_\mathfrak{V}(u),\,\forall a,b,c\in \mathfrak{L},\,u,v,w\in \mathfrak{V}.\label{5.2}
\end{align}

\begin{prop}\label{prop:3-cocycle}
The triple $(\mathfrak{L}\oplus \mathfrak{V},[\cdot, \cdot,\cdot]_\varsigma,d_\varpi)$ is a modified $\lambda$-differential Lie triple system  if and only if
$(\varsigma,\varpi)$ is a 3-cocycle  of the modified $\lambda$-differential Lie triple system $(\mathfrak{L},[\cdot,\cdot,\cdot],d)$ with the coefficient  in $(\mathfrak{V}; \theta, d_\mathfrak{V})$.
\end{prop}
\begin{proof}
The triple $(\mathfrak{L}\oplus \mathfrak{V},[\cdot, \cdot,\cdot]_\varsigma,d_\varpi)$ is a modified $\lambda$-differential Lie triple system if and only if
\begin{align}
&\varsigma(a,b,c)+\varsigma(b,a,c)=0,\nonumber\\
&\varsigma(a,b,c)+\varsigma(c,a,b)+\varsigma(b,c,a)=0,\nonumber\\
&\varsigma(a,b,[x,y,z])+\mathcal{D}(a,b)\varsigma(x,y,z)-\varsigma([a,b,x],y,z)-\varsigma(x,[a,b,y],z)-\varsigma(x,y,[a,b,z])\nonumber\\
 &-\vartheta(y,z)\varsigma(a,b,x)+\vartheta(x,z)\varsigma(a,b,y)-\mathcal{D}(x,y)\varsigma(a,b,z)=0,\label{5.3}\\
&\vartheta(b,c)\varpi(a)+\varsigma(d(a),b,c)-\vartheta(a,c)\varpi(b)+\varsigma(a,d(b),c)+\mathcal{D}(a,b)\varpi(c)+\varsigma(a,b,d(c))\nonumber\\
&+\lambda\varsigma(a,b,c)-\varpi([a,b,c])-d_\mathfrak{V}(\varsigma(a,b,c))=0,\label{5.4}
\end{align}
for any $a,b,c,x,y,z\in \mathfrak{L}$.
Using Eqs. \eqref{5.3} and  \eqref{5.4}, we get $\delta\varsigma=0$ and $\delta\varpi+\Phi\varsigma=0$, respectively.
Therefore, $\partial(\varsigma,\varpi)=(\delta\varsigma,\delta\varpi+\Phi\varsigma)=0,$ that is, $(\varsigma,\varpi)$ is a  3-cocycle.

Conversely, if $(\varsigma,\varpi)$ satisfying Eqs.~\eqref{5.3} and \eqref{5.4}, one can easily check that $(\mathfrak{L}\oplus \mathfrak{V},[\cdot, \cdot,\cdot]_\varsigma,d_\varpi)$ is a modified $\lambda$-differential Lie triple system.
\end{proof}

Next we are ready to classify abelian extensions of a modified $\lambda$-differential Lie triple system.

\begin{theorem}
Abelian extensions of a modified $\lambda$-differential Lie triple system $(\mathfrak{L}, [\cdot, \cdot, \cdot], d)$ by $(\mathfrak{V},[\cdot, \cdot, \cdot]_{\mathfrak{V}}, d_\mathfrak{V})$ are classified by the third cohomology group $\mathcal{H}_{\mathrm{mDLts^{\lambda}}}^3(\mathfrak{L},\mathfrak{V})$ of $(\mathfrak{L}, [\cdot, \cdot, \cdot], d)$ with coefficients in the representation $(\mathfrak{V}; \vartheta, d_\mathfrak{V})$.
\end{theorem}
\begin{proof}
Let $(\hat{ \mathfrak{L}}, [\cdot,\cdot,\cdot]_{\hat{ \mathfrak{L}}}, d_{\hat{ \mathfrak{L}}})$ be an abelian extension of $(\mathfrak{L}, [\cdot, \cdot, \cdot], d)$ by $(\mathfrak{V},[\cdot, \cdot, \cdot]_{\mathfrak{V}}, d_\mathfrak{V})$.  We choose a section $\sigma: \mathfrak{L}\rightarrow \hat{ \mathfrak{L}}$ to obtain a 3-cocycle $(\varsigma,\varpi)$  by Proposition~\ref{prop:3-cocycle}.
First, we show that the cohomological class of $(\varsigma,\varpi)$  is independent of the choice of $\sigma$.
Let $\sigma_1,\sigma_2:\mathfrak{L}\rightarrow \hat{ \mathfrak{L}}$ be two distinct sections providing 3-cocycles $(\varsigma_1,\varpi_1)$ and $(\varsigma_2,\varpi_2)$ respectively. Define
linear map $\xi: \mathfrak{L}\rightarrow \mathfrak{V}$ by $\xi(a)=\sigma_1(a)-\sigma_2(a)$. Then
\begin{align*}
&\varsigma_1(a,b,c)\\
&=[\sigma_1(a), \sigma_1(b),\sigma_1(c)]_{\hat{ \mathfrak{L}}_1}-\sigma_1([a, b, c])\\
&=[\sigma_2(a)+\xi(a), \sigma_2(b)+\xi(b), \sigma_2(c)+\xi(c)]_{\hat{ \mathfrak{L}}_1}-(\sigma_2([a, b,c])+\xi([a, b,c]))\\
&=[\sigma_2(a), \sigma_2(b), \sigma_2(c)]_{\hat{ \mathfrak{L}}_2}+\theta(b,c)\xi(a)-\theta(a,c)\xi(b)+D(a,b)\xi(c)-\sigma_2([a, b,c])-\xi([a, b,c])\\
&=([\sigma_2(a), \sigma_2(b), \sigma_2(c)]_{\hat{ \mathfrak{L}}_2}-\sigma_2([a, b,c]))+\theta(b,c)\xi(a)-\theta(a,c)\xi(b)+D(a,b)\xi(c)-\xi([a, b,c])\\
&=\varsigma_2(a,b,c)+\delta\xi(a,b,c)
\end{align*}
and
\begin{align*}
\varpi_1(x)&=\hat {d}(\sigma_1(a))-\sigma_1(d(a))\\
&=\hat {d}(\sigma_2(a)+\xi(a))-\big(\sigma_2(d(a))+\xi(d(a))\big)\\
&=\big(\hat {d}(\sigma_2(a))-\sigma_2(d(a))\big)+\hat {d}(\xi(a))-\xi(d(a))\\
&=\varpi_2(a)+d_\mathfrak{V}(\xi(a))-\xi(d(a))\\
&=\xi_2(x)-\Phi\xi(a).
\end{align*}
i.e., $(\varsigma_1,\varpi_1)-(\varsigma_2,\varpi_2)=(\delta\xi,-\Phi\xi)=\partial(\xi)\in \mathcal{C}_{\mathrm{mDLts^{\lambda}}}^3(\mathfrak{L},\mathfrak{V})$. So $(\varsigma_1,\varpi_1)$ and $(\varsigma_2,\varpi_2)$ are in the same cohomological class  in $\mathcal{H}_{\mathrm{mDLts^{\lambda}}}^3(\mathfrak{L},\mathfrak{V})$.

Next, assume that $(\hat{\mathfrak{L}}_1, [\cdot,\cdot,\cdot]_{\hat {\mathfrak{L}}_1}, \hat {d_1})$ and $(\hat{\mathfrak{L}}_2, [\cdot,\cdot,\cdot]_{\hat {\mathfrak{L}}_2}, \hat {d_2})$  are two equivalent abelian extensions of $(\mathfrak{L},$
$ [\cdot,\cdot,\cdot], d)$ by $(\mathfrak{V},[\cdot,\cdot,\cdot]_\mathfrak{V}, d_\mathfrak{V})$ with the associated isomorphism $\zeta:(\hat{\mathfrak{L}}_1, [\cdot,\cdot,\cdot]_{\hat {\mathfrak{L}}_1}, $ $\hat {d_1})\rightarrow (\hat{\mathfrak{L}}_2, $
$[\cdot,\cdot,\cdot]_{\hat {\mathfrak{L}}_2}, \hat {d_2})$. Let $\sigma_1$ be a section of $(\hat{\mathfrak{L}}_1, [\cdot,\cdot,\cdot]_{\hat {\mathfrak{L}}_1}, \hat {d_1})$. As $p_2\circ\zeta=p_1$, we  get
$$p_2\circ(\zeta\circ \sigma_1)=p_1\circ \sigma_1= \mathrm{id}_{\mathfrak{L}}.$$
That is, $\zeta\circ \sigma_1$ is a section of $(\hat{\mathfrak{L}}_2, [\cdot,\cdot,\cdot]_{\hat {\mathfrak{L}}_2}, \hat {d_2})$. Denote $\sigma_2:=\zeta\circ \sigma_1$. Since $\zeta$ is a isomorphism of  modified $\lambda$-differential Lie triple systems such that $\zeta|_\mathfrak{V}=\mathrm{id}_\mathfrak{V}$, we have
\begin{align*}
\varsigma_2(a,b,c)&=[\sigma_2(a), \sigma_2(b), \sigma_{2}(c)]_{\hat{\mathfrak{L}}_2}-\sigma_2([a,b,c])\\
&=[\zeta(\sigma_1(a)), \zeta(\sigma_1(b)), \zeta(\sigma_1(c))]_{\hat{\mathfrak{L}}_2}-\zeta(\sigma_1([a, b, c]))\\
&=\zeta\big([\sigma_1(a), \sigma_1(b), \sigma_1(c)]_{\hat{\mathfrak{L}}_1}-\sigma_1([a, b, c])\big)\\
&=\zeta(\varsigma_1(a,b,c))\\
&=\varsigma_1(a,b,c)
\end{align*}
and
\begin{align*}
\varpi_2(a)&=\hat{d_2}(\sigma_2(a))-\sigma_2(d(a))=\hat{d_2}\big(\zeta(\sigma_1(a))\big)-\zeta\big(\sigma_1(d(a))\big)\\
&=\zeta\big(\hat{d_1}(\sigma_1(a))-\sigma_1(d(x))\big)\\
&=\zeta(\varpi_1(a))\\
&=\varpi_1(a).
\end{align*}
Hence, all equivalent abelian extensions give rise to the same element in $\mathcal{H}_{\mathrm{mDLts^{\lambda}}}^3(\mathfrak{L},\mathfrak{V})$.

Conversely, given two  cohomologous 3-cocycles $(\varsigma_1,\varpi_1)$ and $(\varsigma_2,\varpi_2)$  in $\mathcal{H}_{\mathrm{mDLts^{\lambda}}}^3(\mathfrak{L},\mathfrak{V})$, we can construct two abelian extensions $(\mathfrak{L}\oplus \mathfrak{V},[\cdot, \cdot,\cdot]_{\zeta_1},d_{\varpi_1})$ and  $(\mathfrak{L}\oplus \mathfrak{V},[\cdot, \cdot,\cdot]_{\zeta_2},d_{\varpi_2})$ via Eqs.~\eqref{5.1} and \eqref{5.2}. Then, there is  a linear map $\xi: \mathfrak{L}\rightarrow  \mathfrak{V}$ such that
 $$(\varsigma_2,\varpi_2)=(\varsigma_1,\varpi_1)+\partial(\xi).$$
 Define linear map $\zeta_\xi: \mathfrak{L}\oplus \mathfrak{V}\rightarrow  \mathfrak{L}\oplus \mathfrak{V}$ by
$\zeta_\xi(a,u):=a+\xi(a)+u, ~a\in \mathfrak{L}, u\in \mathfrak{V}.$
Then, $\zeta_\xi$ is an isomorphism of these two abelian extensions.
\end{proof}

\noindent
{{\bf Acknowledgments.}  The paper is  supported by the  Foundation of Science and Technology of Guizhou Province(No. [2018]1020),   The NSF of China (No. 12161013).

\end{document}